\documentclass[11pt]{article}

\usepackage{amsmath,amsthm,verbatim,amssymb,amsfonts,amscd, graphicx}
 \usepackage[matrix,arrow,ps]{xy}
 \usepackage[small]{caption}
\usepackage{graphics,pb-diagram}

\theoremstyle{plain}
\newtheorem{theorem}{Theorem}
\newtheorem{corollary}{Corollary}
\newtheorem{lemma}{Lemma}
\newtheorem{proposition}{Proposition}

\theoremstyle{definition}
\newtheorem{definition}{Definition}

\begin{document}

\title{Legendrian grid number one knots and augmentations of their differential algebras}
\author{Joan E. Licata\\
\small\textit{Stanford University}\\
\small\textit{Stanford, CA 94305}\\
\small\textit{jelicata@stanford.edu}}
\maketitle

\abstract{In this article we study the differential graded algebra (DGA) invariant associated to Legendrian knots in tight lens spaces.  Given a grid number one diagram for a knot in $L(p,q)$, we show how to construct a special Lagrangian diagram suitable for computing the DGA invariant for the Legendrian knot specified by the diagram.  We then specialize to $L(p,p-1)$ and show that for two families of knots, the existence of an augmentation of the DGA depends solely on the value of $p$. A version of this article will appear in the Proceedings of the Heidelberg Knot Theory Semester}

\section{Introduction}\label{sect:intro}
Differential graded algebras have been associated to Legendrian knots in a variety of different contact manifolds, including the standard tight $\mathbb{R}^3$, $S^3$, and lens spaces $L(p,q)$ \cite{C}, \cite{S}, \cite{L}.  These algebras may be computed combinatorially from a Lagrangian projection of the knot, and the equivalence class of the algebra is an invariant of Legendrian knot type.  These algebras are a combinatorial interpretation of the relative contact homology developed by Eliashberg, Givental, and Hofer, which is generally difficult to compute \cite{EGH}. In the last few years, attempts have been made to compute these algebras from the front, rather than the Lagrangian, projection, and also to extract more tractable invariants from the DGAs. 

 Although the Lagrangian projection is the most natural for computing the DGA, it has significant drawbacks relative to the front projection.  Lagrangian projections admit only a weak Reidemeister theorem, and it is computationally intensive to determine whether or not a given picture is actually the projection of a Legendrian knot.  In contrast, front projections suffer from neither of these features, but they are less geometrically natural for computing the DGA. For knots in $\mathbb{R}^3$ and the solid torus, this difficulty was addressed by Ng, who used ``resolution" to compute the DGA directly from a front projection \cite{N}.  The results in this paper are part of a program to develop an analogous technique for Legendrian knots in lens spaces.
 
 The front projection of a Legendrian knot in a lens spaces is its projection to a Heegaard torus, and front projections are encoded combinatorially by toroidal grid diagrams as described in \cite{BG}.  As in the Euclidean case, front projections in lens spaces are more easily manipulated than are Lagrangian projections.   We focus on the case of  grid number one knots (defined in Section~\ref{sect:grid}), and we show that for these knots, a grid diagram suffices to determine a labeled Lagrangian projection with a specialized form.  Grid number one knots are of particular interest because of their relationship to the Berge Conjecture, which characterizes knots in $S^3$ which have lens space surgeries \cite{BGH}, \cite{H}, \cite{R}.  Topologically, a grid number one knot is a particular kind of bridge number one knot with respect to a Heegaard torus in $L(p,q)$.  However, we adopt the perspective of \cite{BG}: any grid diagram specifies a particular Legendrian isotopy class within the topological isotopy class.
 
In $\mathbb{R}^3$, Chekanov used the linearized homology of the DGA to distinguish a pair of non-isotopic knots with identical classical invariants \cite{C}. The existence of the linearized homology relies on whether the DGA can be augmented, and the existence of augmentations is itself an invariant of the equivalence type.  Although the existence of augmentations is algorithmically decidable, the computation time is generally exponential in the number of generators.  

As an application of the construction described above, we determine the (non)-existence of augmentations of the DGA for several families of grid number one knots in $L(p,p-1)$.   These theorems, which are stated precisely in the next section, follow in the footsteps of other efforts to detect augmentations of Legendrian DGAs without computing the full differential.  For example,  the relationship between augmentations and rulings of the front projection for knots in $\mathbb{R}^3$ has been extensively studied \cite{Fu}, \cite{FI}, \cite{NS}, \cite{S2}.  Although the current results apply only to particular families of knots, these examples suggest a framework for a more general approach to this problem.

 \subsection{Main results}\label{sect:thms}
 In order to state the main results precisely, we establish some notation that will be used throughout this article.  A grid number one diagram for $L(p,q)$ can be viewed as a row of $p$ boxes, two of which contain basepoints.   For details on how such a diagram specifies a Legendrian knot, see Section~\ref{sect:grid}.  If $s$ is the number of boxes separating the two basepoints, define $v(s)$ by  
 \[s+(v(s))q\equiv 0 \mod  p \text{ and } 1<v(s)<p-1.\] 

Ignoring orientation, the knots associated to a pair of basepoints separated by $s$ or $p-s$ boxes are isotopic, so define 
\[
h= \begin{cases}  s &  \text{ if }s+v(s)<p \\ 
p-s  & \text{ if } s+v(p-s)>p. \end{cases} 
\]

\begin{figure}[h]
\begin{center}
\scalebox{.35}{\includegraphics{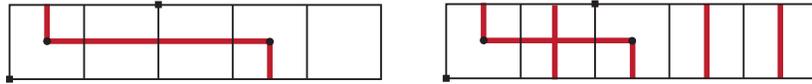}}
\caption{Two diagrams for $K(5,2,3)$.  When $s$ is the length  (measured in boxes) of the horizontal segment connecting the two basepoints, $v(s)$ is the length of the vertical segment in the rectilinear diagram with only southwest and northeast corners.}\label{fig:hv}
\end{center}
\end{figure}

Denote the knot specified by this grid diagram by $K(p,q,h)$.  As shown in \cite{R}, $K(p,q,h)$ is primitive if and only if $gcd(h,p)=1$.  Given a grid number one diagram for a primitive knot in $L(p,q)$ with $q\neq 1$, we construct a Lagrangian diagram for a knot in the associated Legendrian isotopy class.   

Let $\gcd(q-1,p)=k$.  Theorem~\ref{thm:crossings} states that the crossings of the Lagrangian diagram are in one-to-one correspondence with the set of positive integers
 \[ \{ x \ \big| \ x< h  \text{ and } \  k \big| x\} \cup \{ y \ \big| \  y\leq v  \text{ and } k \big| y\}.\] 
 
This shows that if $h, v< k$, then there is a Lagrangian projection of $K(p,q,h)$ with no crossings, so the algebra $\mathcal{A}(K(p,q,h))$ is isomorphic to the ground field $\mathbb{Z}_2$.  

In contrast, the theorem implies that if $gcd(q-1,p)=1$, then there is knot in the specified Legendrian isotopy class whose Lagrangian projection has $h+v+1$ crossings.  It follows that the DGA $\mathcal{A}(K(p,q,h), \partial)$ is a tensor algebra on $2(h+v+1)$ generators, and the ancillary data needed to determine the differential may also be computed from the numerical data associated to the grid diagram (Section~\ref{sect:laggrid} ).  In Section~\ref{sect:augs}, we apply this construction to show that in special cases, the existence of augmentations of the DGA can be deduced directly from the original grid diagram: 

\begin{theorem}\label{thm:main1} 
Let $K=K(p,p-1,1)$ be a Legendrian grid number one knot in $L(p,p-1)$ for $gcd(p-2,p)= 1$. Then the homology of $(\mathcal{A}(K),\partial)$ is a tensor algebra with two generators.  Furthermore, the map sending both generators of $\mathcal{A}(K)$ to $0$ is an augmentation. 
\end{theorem}

\begin{theorem}\label{thm:main2} 
Let $K=K(p,p-1, 2)$ be a Legendrian grid number one knot in $L(p,p-1)$ for  $gcd(p-2,p)= 1$.
Then $(\mathcal{A}(K),\partial)$ has an augmentation if and only if $p\equiv 3 \mod 12$ or $p \equiv 9 \mod 12$.
\end{theorem}

\subsection{Conventions and organization}
Nearly every orientation convention imaginable for lens spaces exists in the literature.  Following \cite{GS}, we view the lens space $L(p,q)$  as the result of $\frac{-p}{q}$ surgery on the unknot in $S^3$.   With this choice, the lens spaces $L(p,1)$ are smooth $S^1$ bundles over $S^2$, and the combinatorial formulation of the DGA for Legendrian knots in these spaces is due to Sabloff \cite{S}.  The invariant described in Section~\ref{sect:dga} applies to $L(p,q)$ with $q\neq 1$.    Furthermore, we make use of the canonical correspondence between grid diagrams and toroidal front projections which is described in \cite{BG}, but our use of ``grid diagram" agrees with the authors' use of ``dual grid diagram" in \cite{BG}.  Throughout the paper, the ground field is $\mathbb{Z}_2$.
 
The next section has a brief introduction to augmentations, the universally tight contact structure on lens spaces,  grid diagrams, and the Lagrangian DGA for knots in lens spaces.  In Section~\ref{sect:special} we describe how to construct the Lagrangian projection of a knot from a grid number one diagram.  Finally, in Section~\ref{sect:laggrid} we apply this construction to special classes of knots in $L(p,p-1)$ and prove Theorems~\ref{thm:main1} and \ref{thm:main2}.

\subsection{Acknowledgement} I would like to thank the Max Planck Institute for their support and hospitality during the 2008-2009 year.

\section{Background}
This section briefly reviews augmentations, grid diagrams, and the DGA for primitive knots in lens spaces.  We refer the reader to \cite{N}, \cite{BGH}, and \cite{L} for more details.  

\subsection{Contact lens spaces}\label{sect:cls}
View $S^3$ as the unit sphere in $\mathbb{C}^2$ with polar coordinates:
\[S^3=\{ (r_1, \theta_1, r_2, \theta_2) | r_1^2+r_2^2=1\}\ .\]
These coordinates suggest a genus one Heegaard splitting of $S^3$ along the torus $r_1=r_2=\frac{1}{\sqrt{2}}$.  Define $F_{p,q}:S^3\rightarrow S^3$ by 
\[F_{p,q}(r_1, \theta_1, r_2, \theta_2)=(r_1, \theta_1+\frac{2\pi q}{p}, r_2, \theta_2+\frac{2\pi }{p}).\] 
The map $F_{p,q}$ is an automorphism of $S^3$ with order $p$, and the quotient of $S^3$ by the equivalence induced by $F_{p,q}$ is the lens space $L(p,q)$.  Since $F_{p,q}$ preserves the $r_i$ coordinates in $S^3$, the lens space inherits a genus one Heegaard splitting whose core curves $C_1$ and $C_2$ are the images of the curves $r_1=0$ and $r_2=0$ in $S^3$.

The standard tight contact structure on $S^3$ is induced by the one-form
\[r_1^2d\theta_1+r_2^2d\theta_2 . \]
For $0<r_1<1$, the Reeb vector field is  the constant vector field $<1,1,0>$ with respect to the basis $\{d\theta_1, d\theta_2, dr_1\}$ on $T(T^2\times (0,1))$.  Thus the set of Reeb orbits consists of curves with slope $\frac{d\theta_2}{d\theta_1}=1$ on each torus of fixed $r_1$, together with the two core curves $r_1=0$ and $r_2=0$.  The \textit{Lagrangian projection} of $S^3$ is the orbit space of $S^3$ as an $S^1$ bundle over $S^2$.  

The map $F_{p,q}$ preserves the contact structure on $S^3$, so $L(p,q)$ inherits a tight contact structure from its universal cover.  Throughout this article, we will assume that $L(p,q)$ is equipped with this universally tight contact structure, and we suppress it from the notation.  The Lagrangian projection of $L(p,q)$ is again $S^2$, and when $q=1$, this contact form induces a smooth $S^1$ bundle structure on $L(p,1)$. In contrast, when $q>1$, the images of the core curves $C_1$ and $C_2$ are orbifold points in the Lagrangian projection.  In this case, the covering map $S^3\rightarrow L(p,q)$ factors as a composition of cyclic covers  \[S^3\rightarrow L(k,1)\rightarrow L(p,q),\] where $k=\gcd(q-1,p)$.  The ramified points of the Lagrangian projection  of $L(p,q)$  will have order $\frac{p}{k}$, and we identify these with the south and north poles of $S^2$. We will always assume that the knot $K$ lives in $L(p,q)-(C_1\cup C_2)$.

\subsection{Grid diagrams}\label{sect:grid}

The \textit{front projection} of a Legendrian knot $K \subset L(p,q)$ is its projection to the genus-one Heegaard surface inherited from $S^3$.  The knot $K$ is completely determined by its front projection, and the Legendrian isotopy class of $K$ may may be encoded combinatorially by a grid diagram. In this section we introduce grid diagrams for links in $S^3$ and grid number one diagrams for knots in $L(p,q)$.  The relationship between a knot $K\subset L(p,q)$ and its $F_{p,q}$-preimage $\widetilde{K}\subset S^3$ plays an important role in the construction of the DGA $(\mathcal{A}(K), \partial)$, and grid diagrams are a useful tool for understanding this relationship. 

In $S^3$, parameterize the $r_1=\frac{1}{\sqrt{2}}$ torus  by $\theta_i$ coordinates, where $0\leq \theta_i <2\pi$ for $i=1,2$.  Decorate this torus with curves satisfying $\theta_i=\frac{2n\pi}{p}$ for $n\in \{ 0, 1, ...p-1\}$ and $i=1,2$.  Cutting the torus along the curves $\theta_i=0$ yields a square divided into $p^2$ boxes which are arrayed in $p$ rows and $p$ columns.  Although the diagram is drawn as a planar object, we retain the identifications $(t,0)\sim (t, 1)$ and $(0, t)\sim (1,t)$ in order to simultaneously view the planar grid diagram as a decorated Heegaard torus. Finally, add $2p$ basepoints to the diagram, so that each column and each row contains exactly two basepoints. The decorated square is called a \textit{grid number $p$ grid diagram}. 

\begin{figure}[h]
\begin{center}
\includegraphics[scale=.25]{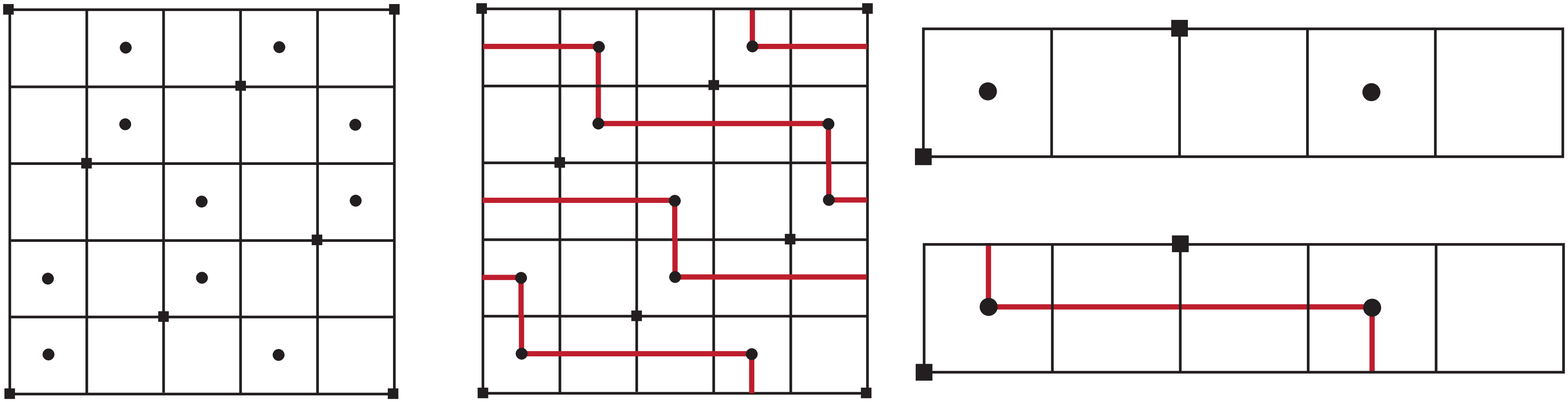}
\caption{Left: An $F_{5,2}$-invariant grid number $5$ diagram for $\widetilde{K}\subset S^3$.  Center: A compatible rectilinear diagram for $\widetilde{K}\subset S^3$. Right : A grid number one diagram for $K(5,2,3)$, together with a compatible rectilinear projection.  In the grid number one diagrams, the rectangles on the boundary indicate the gluing which yields a Heegaard torus for $L(5,2)$.}
\end{center}
\label{fig:stack}
\end{figure}

If a  grid number $p$ grid diagram $\widetilde{\Sigma}$ is invariant under $F_{p,q}$, then its quotient under the equivalence induced by $F_{p,q}$ is called a \textit{grid number one grid diagram} $\Sigma$. The bottom row of $\widetilde{\Sigma}$ is a fundamental domain for the action of $F_{p,q}$, so we may parameterize $\Sigma$ by $\theta_i$ coordinates  with $0\leq \theta_1 <2\pi$ and $0\leq \theta_2 <\frac{2\pi}{p}$.   As in the case of a grid number $p$ diagram in $S^3$, a grid number one diagram is a planar depiction of a torus; to recover  a Heegaard torus for $L(p,q)$, identify $(t, 0)$ with $(t+\frac{2q\pi}{p}, \frac{2\pi}{p})$ and $(0, t)$ with $(1,t)$.  These identifications yield a torus decorated by two connected curves which intersect $p$ times.  The complement of the images of the vertical curves from $\widetilde{\Sigma}$ is connected in $\Sigma$, and it is referred to as the \textit{column} of the grid diagram.  Similarly, the complement of the images of the horizontal curves is referred to as the \textit{row} of the grid diagram.   
 
A grid diagram specifies a link in the associated three-manifold.  Connect each pair of basepoints in the same row or same column by a linear segment.  Observe that for each such pair of basepoints, one may choose between two possible line segments on the torus.  The result of any set of choices is called a \textit{rectilinear diagram compatible with the grid diagram}.  (See Figure~\ref{fig:stack}.)  Viewing the rectilinear diagram as a curve in the three-manifold, push the interior of each horizontal curve into the solid torus defined by $r_1> \frac{1}{\sqrt{2}}$ and push the interior of each vertical curve into the solid torus defined by $r_1<\frac{1}{\sqrt{2}}$.  The resulting embedded curve intersects the Heegaard torus only at the original basepoints of the grid diagram. 

The topological isotopy class of the knot constructed this way is independent of the choice of rectilinear projection, so one may refer to a \textit{grid diagram for a knot $K$}.  Given a grid diagram $\Sigma$ for $K\subset L(p,q)$, one may construct a grid diagram for its $F_{p,q}$-preimage $\widetilde{K}\subset S^3$ using $p$ copies of $\Sigma$.  This construction is indicated in Figure~\ref{fig:stack}, and we refer the reader to \cite{BGH} for a fuller treatment.  Following Rasmussen, we say that a knot in $L(p,q)$ is \textit{primitive} if it generates $H_1(L(p,q))$. The knot $K\subset L(p,q)$ is primitive if and only if $\widetilde{K}\subset S^3$ has one component, and in the notation from Section~\ref{sect:thms}, $K(p,q,h)$ is primitive if and only if $h$ and $p$ are relatively prime  \cite{R}.

A grid diagram can be interpreted as specifying not simply a topological isotopy class of knot, but rather a Legendrian isotopy class of knot in the standard contact $S^3$ or $L(p,q)$ \cite{BG}, \cite{OSzT}.  Proposition 3.3 of \cite{BG} asserts that any curve on $\Sigma$ with $\frac{d\theta_2}{d\theta_1}$ slope in $(-\infty, 0)$ which is smoothly embedded away from semi-cubical cusps or transverse double points is the front projection of some Legendrian knot in $L(p,q)-(C_1 \cup C_2)$.  The knot can be recovered from this projection because the Legendrian condition implies that the slope of the front projection of $K$ determines the $r_1$ coordinate: \[\frac{d\theta_2}{d\theta_1}=\frac{-r_1^2}{1-r_1^2}.\]
  
 \begin{theorem}[\cite{BG}]\label{thm:bg} Any rectilinear diagram compatible with a grid diagram is isotopic on $\Sigma$ to the front projection of some Legendrian knot $K$ in $L(p,q)$.  The Legendrian isotopy class of $K$ is independent of the choice of compatible rectilinear diagram.
\end{theorem}

\subsection{Differential graded algebras}
Let $V$ be the vector space generated by  $\{v_i\}_{i=1}^n$.  Then the \textit{tensor algebra} generated by the $v_i$ is
\[ 
T(v_1, ...v_n)=\bigoplus_{k=0}^{\infty} V^{\otimes k}
\]
If $V$ is graded by some cyclic group,  extend the grading by setting $|v_1v_2|=|v_1|+|v_2|$.  If $\partial:\mathcal{A}\rightarrow\mathcal{A}$ is a graded degree $-1$ map which satisfies $\partial^2=0$, then the pair $(\mathcal{A}, \partial)$ is a \textit{semi-free differential graded algebra} (DGA).  

\begin{definition} An \textit{augmentation} of a DGA is a graded algebra homomorphism $\epsilon:\mathcal{A}\rightarrow \mathbb{Z}_2$ such that $\epsilon(1)=1$, $\epsilon \circ \partial=0$ and $\epsilon(x)=0$ if $|x|\neq 0$.  
\end{definition}

The natural notion of equivalence on DGAs is that of stable tame isomorphism.  (For a definition, see \cite{C}, \cite{S}, or \cite{L}.)  Equivalent DGAs have isomorphic homology, and the existence of augmentations of a DGA $(\mathcal{A}, \partial)$ is also an invariant of its equivalence type.  Furthermore, the number of augmentations is an invariant of the equivalence type, up to a power of two.  

\subsection{The Lagrangian DGA for $K\subset L(p,q)$}\label{sect:dga}
In this section we introduce the Lagrangian DGA for primitive Legendrian knots in lens spaces $L(p,q)$ with $q\neq 1$.  The algebra is generated by Reeb chords with both endpoints on $K$, so each crossing in the Lagrangian projection of $K$ corresponds to a pair of complementary Reeb chords in $L(p,q)$.  

The DGA $(\mathcal{A}(K), \partial)$ is defined in purely combinatorial terms from a labeled Lagrangian projection of $K$, but the reader may find it helpful to know that the proof of invariance relies on the relationship between $K\subset L(p,q)$ and a cyclic cover $\widetilde{K}\subset M_k$.  If $\gcd(q-1,p)=1$, the cyclic cover $M_k$  is the universal cover of $L(p,q)$.  If $\gcd(q-1,p)=k>1$, then $M_k$ is the lens space $L(k,1)$, which is a $\frac{p}{k}$-to-one cover of $L(p,q)$. The results in this paper are stated for primitive knots in $L(p,q)$, but  they  generalize to any knot which generates an order $\frac{p}{k}$ subgroup of $\pi_1\big(L(p,q)\big)$. For a fuller description of this relationship, see \cite{L}. 

Given $K\subset L(p,q)$, denote by $\Gamma$  the Lagrangian projection of $K$. The preimage of each double point in $\Gamma$ consists of a pair of complementary  Reeb chords in $L(p,q)$, and we (arbitrarily) designate one chord in each pair as \textit{preferred}.  At each crossing, $\Gamma$ divides a neighborhood of the crossing into four quadrants, and we will label these quadrants so as to identify the preferred chord.  (See Figure~\ref{fig:crossing}.) Because the Reeb orbits are integral curves of the Reeb vector field, the Reeb chords are naturally oriented.  At a fixed crossing, each oriented chord $x$ assigns ``source" and ``sink" labels to the two arcs of $K$ which project to the crossing, and we label a quadrant by ``$x^+$" if traveling along the source curve to the sink curve in $\Gamma$ orients the quadrant they bound positively.  Similarly, we label a quadrant  $x^-$ if  traveling source-to-sink orients the bounded quadrant negatively.   If $x$ is the preferred generator at a crossing, then we decorate the $x^+$ quadrants with an additional ``\textbf{+}".  

\begin{figure}[h]
\begin{center}
\includegraphics[scale=.55]{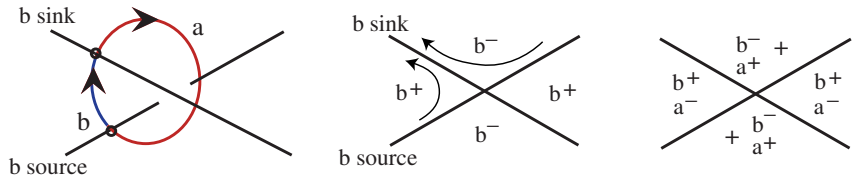}
\caption{The orientation of the chords determines the quadrant labels in the Lagrangian projection.  The left figures shows two strands of $K$ intersecting a fixed Reeb orbit, and the other diagrams indicated the associated labeling in the Lagrangian projection.  The plus signs in the right figure indicate that $a$ is the preferred chord.}
\end{center}
\label{fig:crossing}
\end{figure}

Suppose that  the Lagrangian projection of $K$ has $m$  crossings, and let $a_i$ and $b_i$ be the complementary Reeb chords associated to the $i^{th}$ crossing of $\Gamma$.  We  define $\mathcal{A}(K)$ to be the tensor algebra generated over $\mathbb{Z}_2$ by the $\{a_i\}$ and $\{b_i\}$: 
\[\mathcal{A}(K)=T(a_1,b_1, ...a_m, b_m).\]

The remainder of this section is devoted to defining the boundary map $\partial: \mathcal{A}(K)\rightarrow \mathcal{A}(K)$.  In order to do so, we will associate a rational-valued \textit{defect} to each component of $S^2 - \Gamma$.  A Lagrangian projection which is decorated with defects in each region and ``\textbf{+}" signs to denote the preferred chords at each crossing  is called a \textit{labeled Lagrangian diagram} for $K$. We defer a description of the grading on $\mathcal{A}(K)$ until Section~\ref{sect:grad}, as it is not necessary to understand Section~\ref{sect:special}.

\subsubsection{ The defect}

We will simultaneously think of the generators of $\mathcal{A}(K)$ as formal symbols and as Reeb chords in $L(p,q)$.  As such, we can assign a \textit{length} to each generator of the algebra:

\begin{definition} The  \textit{length} of a generator $x$ is the length of the associated Reeb chord, measured as a fraction of the orbital period.  Denote the length of $x$ by $\mathit{l}(x)\in (0,1)$.
\end{definition}

As above, let $\gcd(q-1,p)=k$, and let $M_k$ denote the $\frac{p}{k}$-fold cyclic cover of $L(p,q)$.  The manifold $M_k$ is either $L(k,1)$, if $k>1$, or $S^3$, if $k=1$. The contact form on $M_k$ induces a curvature $2$-form on its Lagrangian base space $S^2$.   This copy of $S^2$ is a $\frac{p}{k}$-to-one branched cover of the Lagrangian base space of $L(p,q)$, and the latter two-sphere inherits an area form from its covering space.  Normalize the induced form so that the Lagrangian projection of $L(p,q)$ has area equal to $\frac{k^2}{p}$.  Let $a(R)$ denote the area of a region $R\in S^2-\Gamma$.

Suppose that $x_m$ is  the preferred generator at a crossing where the region $R$ has a corner. If $R$ fills the quadrant labeled $x_m^+$, define $\epsilon(m)=1$, and  if $R$ fills the quadrant labeled $x_m^-$, define $\epsilon(m)=-1$.

\begin{definition}\label{def:defect} Let $R$ be a component of $S^2-\Gamma$ with $m$ corners. The \textit{defect}  of $R$ is given by  
\[ n(R)=-a(R)+\sum_{i=1}^m \epsilon(i) \mathit{l}(x_i),\]
where the sum is taken over the preferred generators at crossings where $R$ has a corner.
\end{definition}

Definition~\ref{def:defect} is closely related to Sabloff's definition of defect for Legendrian knots in smooth lens spaces, and we offer the following geometric perspective on $n(R)$ \cite{S}. The boundary of $R$ lifts to a curve $ \gamma \in L(p,q)$ which is composed of alternating Legendrian segements and  preferred Reeb chords.  If $R$ is disjoint from the poles of $S^2$,  the defect of $R$ is the winding number of $\gamma$ around the Reeb orbit with respect to an appropriate trivialization of the $S^1$ bundle.  If $R$ contains one of the poles of $S^2$, then it lifts to a region $\widetilde{R}$ in the Lagrangian projection of $(M_k, \widetilde{K})$. One may similarly associate a winding number $n(\widetilde{R})$ to this region, and $n(R)=\frac{k}{p}n(\widetilde{R})$.  It follows from \cite{S} and \cite{L} that a region disjoint from the poles will have integral defect, and the defect of a region containing a pole will lie in  $\frac{k}{p}\mathbb{Z}$.

If $f:(D^2, \partial D^2)\rightarrow (S^2, \Gamma)$ is smooth on the interior of $D^2$, extend the definition of defect to $n(f(D^2))$ by summing the defects of the regions in $f(D^2)$, counted with multiplicity.  
 
\noindent\textbf{Remark:} In \cite{L}, the defect is defined in terms of the lift of $K$ to $\widetilde{K}\subset M_k$.  The area term in the present definition replaces the curvature term seen there, and the equivalence of the definitions follows from the identification of the curvature form on $S^2$ with the Euler class of $M_k$ as a unit sphere bundle \cite{G}.  
 
\subsubsection{The boundary map}
 
\begin{definition} An \textit{admissible disc} is a map $f:(D^2, \partial D^2)\rightarrow (S^2, \Gamma)$ of the disc with $m$ marked points  on its boundary which satisfies the following properties:
\begin{enumerate}
\item\label{cond:imm} either $f$ is an immersion or $f$ fails to be an immersion only at points which map to the poles of $S^2$.  In the latter case, $f$ is diffeomorphic to $z\rightarrow z^{\frac{p}{k}}$ in a neighborhood of each singular point;
\item each marked point maps to a crossing of $\Gamma$;.
\item $f$ extends smoothly to $\partial D^2$ away from the marked points;
\item at each marked point, $f(D^2)$ fills exactly one quadrant.\end{enumerate}
\end{definition}

Two admissible discs $f$ and $g$ are \textit{equivalent} if there is a smooth automorphism $\phi: D^2 \rightarrow D^2$ such that $f=g\circ \phi$.  As above, the defect of an admissible disc $f$ is the sum of the defects of regions in its image, counted with multiplicity.  

If $f$ is an admissible disc which fills a quadrant marked with $x^+$, we associate to $f$ the \textit{boundary word} $w(f,x)$ and the \textit{$x$-defect} $n_x(f)$:
\begin{itemize}
\item Moving counterclockwise around $\partial D^2$ from the point mapping to $x^+$, let $y_i^-$ be the negative generator labels associated to the $i^{th}$ corner of the image of $D^2$, where $1\leq i\leq m-1$.  Then  $w(f,x)=y_1y_2...y_{m-1}$. See Figure~\ref{fig:word} for an example
\item The $x$-defect $n_x(f)$ is computed from the defect of $f(D^2)$ by subtracting one for each $y_i$ which is not a preferred generator and adding one if $x$ is not a preferred generator.
\end{itemize}

\begin{definition} Define $\partial:\mathcal{A}\rightarrow \mathcal{A}$ on generators of $\mathcal{A}$ by
\[\partial (x)=\sum_{f:n_x(f)=0} w(f,x).\]
Extend $\partial$ to all of $\mathcal{A}$ via the Leibniz rule $\partial (ab)=(\partial a)b+a(\partial b)$.
\end{definition}

\begin{figure}[h]
\begin{center}
\includegraphics[scale=.6]{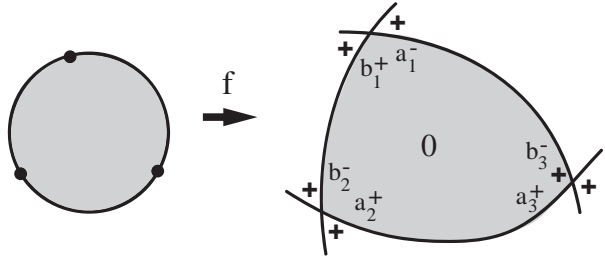}
\caption{The $a_2$-defect of $f$ is $n(f(D^2))+1-1=n(f(D^2))=0$, so the disc shown contributes the term $w(f,a_2)=b_3a_1$ to $\partial a_2$.   The same disc $f$ also corresponds to $w(f,a_2)=a_1b_3 \in \partial a_2$ and $w(f,b_1)=b_3b_2 \in \partial b_1$.}
\end{center}
\label{fig:word}
\end{figure}

\begin{theorem}[\cite{L}] If $K$ is a primitive Legendrian knot in $L(p,q)$ for $q\neq 1$, the pair $(\mathcal{A}(K), \partial)$ is a DGA.  The equivalence type of this DGA is an invariant of the Legendrian isotopy class of $K$.
\end{theorem}

\section{Converting fronts to Lagrangian projections}\label{sect:special}   

Given a grid diagram, Section~\ref{sect:grid} describes how to construct a curve which is the front projection of a knot in the Legendrian isotopy class indicated by the grid diagram.  This front projection completely determines the knot, as the slope of the curve at each point recovers the coordinate lost in the projection.  Abstractly, this implies that the grid diagram carries sufficient information to compute the DGA of the associated Legendrian knot, but the conversion from grid diagram to front projection to labeled Lagrangian diagram may be difficult in practice.  Furthermore, although the grid diagram determines a unique Legendrian isotopy class of knot, isotopic front projections may correspond to knots whose Lagrangian projections vary tremendously; the choice of front projection therefore greatly affects the computability of the DGA.  In this section we describe how to construct a relatively simple Lagrangian projection directly from a grid diagram for a Legendrian knot in $L(p,q)$.

Our approach is as follows: beginning with a grid number one diagram for $K\subset L(p,q)$, draw a special front projection compatible with the grid diagram.  This front projection represents the choice of a fixed knot $K_0$ in the Legendrian isotopy class determined by the grid diagram.  We parameterize the grid diagram for $K$ in $(\theta_1, \theta_2)$ coordinates, where $0\leq \theta_1< 2\pi$ and $0\leq \theta_2< \frac{2\pi}{p}$.  The Lagrangian projection of $L(p,q)$ is a two-sphere, and  we parameterize this by $(\phi, r_1)$, where $-\pi \leq \phi<\pi$ is the azimuthal coordinate.  The $r_1$ coordinate corresponds to latitude on the sphere, and the north and south poles are the two ramified points of the covering map between the Lagrangian projections of $L(p,q)$ and its cyclic cover $M_k$. We will find it convenient to represent this $S^2$ as the rectangle $[-\pi, \pi]\times [0,1]$, and we recover the sphere via the identification of each of the top and bottom edges to a point, as well as the further identification $(-\pi, t)\sim (\pi, t)$, which glues the left and right edges of the square.

In order to move between  different projections, parameterize the Lagrangian projection of $L(p,q)$ so that the Reeb orbit on the Heegaard torus which passes through the point $\theta_1=\theta_2=0$ maps to $\phi=0$.  

\begin{lemma}\label{lemma:coord} The $\theta_i$ and $\phi$ coordinates are related by the equation 
\[ \phi=\frac{p}{k}(\theta_1-\theta_2) \mod 2\pi.\] \end{lemma} 

\begin{proof} 
Begin by considering a grid diagram for $S^3$, and recall that each Reeb orbit on the Heegaard torus is a curve with constant slope $\frac{d\theta_2}{d\theta_1}=1$.  Thus, holding $\theta_1-\theta_2$ fixed determines a Reeb orbit, and consquently, a point in the Lagrangian projection of $S^3$.  

Now recall that the action of $F_{p,q}$ on $S^3$ permutes a set of $\frac{p}{k}$ distinct Reeb orbits.  The Reeb orbit on the Heegaard torus which passes through $(0,0)$ is identified in the quotient with the Reeb orbits passing through the points $\big(\frac{2ck\pi}{p}, 0\big)$ for $c\in \mathbb{Z}_k$.  Since the $r_1$ coordinate is independent of $\phi$ and the $\theta_i$, the formula follows.  

\end{proof}

Figure~\ref{fig:phi} provides an example illustrating Lemma~\ref{lemma:coord}
  
 The remainder of this section is divided into two parts.  In the first, we describe the special front more precisely.   In Section~\ref{sect:laggrid}, we consider a knot $K_0$ with a special front projection and we show how features of the Lagrangian and front projections of $K_0$ correspond. 
 
 \subsection{Special front projections}
 Label the boxes of the grid diagram from $0$ to $p-1$ so that one basepoint appears in Box $0$ and the other in Box $h$.  Connect the two basepoints by a horizontal line whose length (measured in boxes) is $h$.  Moving downward from the basepoint in Box $h$, draw a vertical line in the column of the grid diagram which connects the two basepoints and has length $v$.  Add a new basepoint to the curve each time it passes through the center of a box until the curve consists of $h+v$ equal-length segments that meet at basepoints.  Now allow each basepoint to slide along the anti-diagonal of its box so that the slopes of the line segments connecting successive pairs decrease strictly as one moves right from the basepoint in Box $0$. 
 
Observe that by keeping these perturbations small, the slopes of the formerly-horizontal segments can be held arbitrarily close to $0$ and the slopes of the formerly-vertical segments can be held arbitrarily close to $-\infty$.  Finally, replace a neighborhood of each basepoint with a curve that smoothly and strictly monotonically interpolates between the slopes of the line segments to either side.  Since the resulting curve is smooth with negative slope, Proposition 3.3 of \cite{BG} implies that it is the front projection of a Legendrian knot $K$ in $L(p,q)$.

\begin{figure}[h]
\begin{center}
\includegraphics[scale=.35]{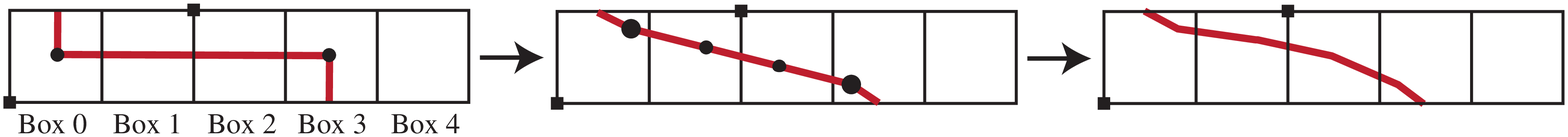}
\caption{Transforming a rectilinear diagram for $K(5,2,3)$ into a special front projection}
\end{center}
\label{pic2}
\end{figure}
 
\subsection{The Lagrangian projection associated to a special front}\label{sect:laggrid}
 
 Let $K_0$ be a Legendrian knot in $L(p,q)$ whose front projection has the form described in the previous section.  For convenience, we will continue to describe the line segments as connecting basepoints except when specifically focusing on the short connecting curves introduce in the previous paragraph.  Each of these line segments on the front corresponds to a subcurve of $K_0$ with a fixed $r_1$ coordinate.  Thus, each subcurve maps to a horizontal curve on the Lagrangian projection.  Since each of the line segments in the special front projection has a distinct slope, the corresponding horizontal curves in the Lagrangian projection are disjoint. Furthermore,  because each line segment connects basepoints in adjacent boxes on the grid diagram, the Lagrangian projection of the corresponding subcurve satisfies \[\frac{2c\pi}{k} +\epsilon' <\phi< \frac{2(c+1)\pi}{k}-\epsilon'\] for some $c \in \mathbb{Z}_k$.  
 
 Each connecting curve on the special front projection lies in a small neighborhood of the anti-diagonal of some box of the grid diagram, so the Lagrangian projection of the corresponding subcurve of $K_0$ lies in a neighborhood of one of the vertical lines $\phi=\frac{2c\pi}{k}$ for $c\in \mathbb{Z}_k$.  The connecting curve in Box $0$ joins the line segment with the most negative slope to the line segment with the least negative slope; the image of this curve in the Lagrangian projection joins the $\phi=-\epsilon'$ endpoint of the bottom horizontal curve to the $\phi=\epsilon'$ endpoint of the top horizontal curve.  We will refer to this as the \textit{ascending curve}. Each of the other connecting curves on the special front joins the right endpoint of a segment to the left endpoint of a segment with a more negative slope; the corresponding \textit{descending curve} on the Lagrangian projection joins the $\phi=\frac{2c\pi}{k}-\epsilon'$ endpoint of a horizontal line to the $\phi=\frac{2c\pi}{k}+\epsilon'$ endpoint of the horizontal line immediately below, for some $c\in \mathbb{Z}_k$.  
  
  \begin{figure}[h]
  \begin{center}
\includegraphics[scale=.35]{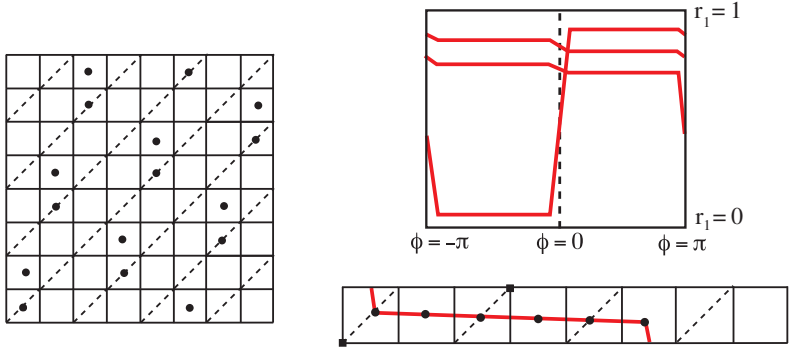}
\caption{ Left: An $F_{8,3}$-invariant grid diagram for a knot in $S^3$.  The dotted lines show $4$ Reeb orbits which are identified in the quotient.  Right: A special front projection of $K(8,3,5)$, together with a schematic Lagrangian projection.}
\end{center}
\label{fig:phi}
\end{figure}
  
 From this description, we can extract the number of crossings of $\Gamma$ and the number of connected components of $S^2-\Gamma$:
  
 \begin{theorem}\label{thm:crossings} The crossings of $\Gamma$ are in one-to-one correspondence with the set of positive integers \[ \{ x \ \big| \ x< h  \text{ and } \  k \big| x\} \cup \{ y \ \big| \  y\leq v  \text{ and } k \big| y\}.\]
 The number of connected components of $S^2-\Gamma$ is two more than the number of crossings.
 \end{theorem} 
 
 \begin{corollary} When $gcd(q-1,p)=1$, the Lagrangian projection of $K(p,q,h)$ has $h+v-1$ crossings,  and there are $h+v+1$ connected components of $S^2-\Gamma$.
 \end{corollary}
 
  \begin{corollary} When $h,v<\gcd(q-1,p)$, the Lagrangian projection of $K(p,q,h)$ has no crossings.  
 \end{corollary}

 \begin{proof}[Proof of Theorem~\ref{thm:crossings}]
  
  In the complement of the ascending curve, we may parameterize  $\Gamma$ so that $dr_1\leq 0$  and $d\phi>0$.  Thus, this portion of the Lagrangian projection of $K_0$ embeds in $S^2$ as a descending spiral.   In order to determine the crossings of $\Gamma$, we count the number of times the ascending curve crosses this spiral.   Observe that the ascending curve lies in a neighborhood of the line $\phi=0$.  Each crossing in the Lagrangian projection corresponds to a point on the spiral with $\phi=0$, and this in turn corresponds to a connecting segment on the special front which appears in a box numbered $ck$ for $c\in \mathbb{Z}_{\frac{p}{k}}$. The indexing set in the statement of Theorem~\ref{thm:crossings} counts the number of times the front projection of $K_0$ passes through a box whose label is divisible by $k$.  If $\Gamma$ has no crossings, it divides $S^2$ into two regions; each time $\Gamma$ spirals around $S^2$  adds a new crossing and a new complementary region.
  
\end{proof}

 The proof of Theorem~\ref{thm:crossings} suggests a fuller description of the Lagrangian projection of $K_0$. If the diagram has no crossings, then the DGA is isomorphic to $\mathbb{Z}_2$, so consider the case when $\Gamma$ has at least one crossing.  Each of the regions of $S^2-\Gamma$ which contains a pole has a single corner. Generically, each of the remaining regions of $S^2-\Gamma$ has four corners, but this number is reduced by one for each adjacent polar region.  It is also possible to completely label the Lagrangian diagram using data from the front projection.  
  
As described in Section~\ref{sect:dga}, each crossing in $\Gamma$ corresponds to a pair of Reeb chords of $K_0 \in L(p,q)$ and, consequently, a pair of generators of $\mathcal{A}(K_0)$.  We number the intersections of the Lagrangian diagram, counting from the top down.  A neighborhood of each crossing is divided into north, east, south, and west quadrants by $\Gamma$.  Define $a_i$ to be the generator which corresponds to a ``$+$" label on the north and south quadrants of the $i^{th}$ crossing.  Similarly, let  $b_i$ be the generator which corresponds to a ``$+$" label on the east and west quadrants of the $i^{th}$ crossing. We will refer to the generators $a_i$ as \textit{$a$-type} generators, and at each crossing, designate the $a$-type generator as preferred. 
 
\begin{figure}[h]
\begin{center}
\includegraphics[scale=.5]{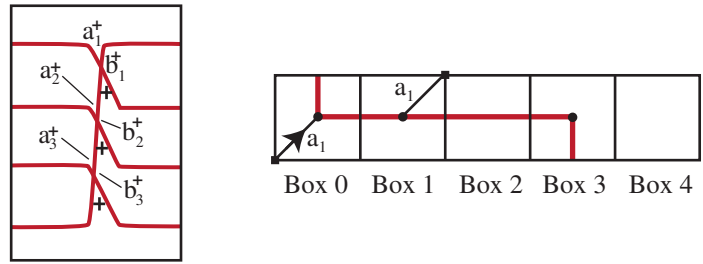}
\caption{Left: A schematic Lagrangian projection of $K(5, 2, 3)$, partially labeled. Right: Image of the Reeb chord $a_1$  on the front projection, which shows that $\mathit{l}(a_1)=\frac{1}{5}$.}
\end{center}
\label{pic3}
\end{figure}

Assigning a defect to each region requires the lengths of the generators and the areas of the regions, both of which may be computed from a  grid diagram.  To compute the area of a given region, recall that the total area of the Lagrangian projection of $L(p,q)$ is normalized to be $\frac{k^2}{p}$.  Assume that each connecting segment on the front projection of $K_0$ lies in small neighborhood of the center of its box on the grid diagram.  This forces the constant-$r_1$ curves on the Lagrangian projection to lie in a small neighborhood of the two poles of $S^2$.  The descending curve corresponding to the basepoint in Box $h$ travels from a neighborhood of the north pole to a neighborhood of the south pole, whereas each of the other descending curves remains in an neighborhood of one of the poles. As a consequence, only the regions bounded to the left and right by the Box $h$ descending curve will have more than negligible area.  When $k=1$, these two regions coincide and the unique large region has area approximately $\frac{1}{p}$.  When $k>1$, each of the two large regions has area approximately equal to some integral multiple of $\frac{k}{p}$.  

 To compute the length associated to each generator, it will be convenient to identify Reeb chords in $L(p,q)$ with their images on the Heegaard torus.  In particular, we will denote by $a_i$ either the Reeb chord in $L(p,q)$ or its projection to the Heegaard torus.  Each Reeb chord has the same length as its front projection, and we may compute that latter by counting boxes in the grid diagram.  The following proposition gives a formula for computing the length of the preferred chords. 
 
\begin{proposition}\label{prop:length} If $s$ is the total number of crossings in the Lagrangian projection of $K_0$, define $B(j)$ by 
\[B(j)= \begin{cases}  kj &  \text{ if }j\leq \frac{h}{k} \\ 
(-qk)(s+1-j) \mod p & \text{ if } j>\frac{h}{k}. \end{cases} 
\] 
If $x_j$ denotes the least positive integer such that
\[B(j)+(1-q)x_j \equiv 0 \mod p,\]
then the length of the generator $a_j$ is $\epsilon$-close to $\frac{k}{p}x_j$.  Furthermore,  $\mathit{l}(b_j)=1-\mathit{l}(a_j)$.
\end{proposition}

\begin{proof}
As noted above, each crossing occurs at a point where the ascending curve in the Lagrangian projection intersects a descending curve.  The formula for $B(j)$ converts between two numbering systems: numbering a descending connecting segment by the crossing it projects to in $\Gamma$ (its $j$ label) and numbering it by the box it lies in on the grid diagram (its $B(j)$ label).  The  endpoints of the chord $a_j$ front-project to basepoints in Box $0$ and Box $B(j)$, and the orientation convention described above implies that $a_j$ is oriented from Box $B(j)$ to Box $0$. For $j\leq \frac{h}{k}$, each increase in $j$ corresponds to traveling $k$ boxes to the right, which increases $B(j)$ by $k$.  For larger values of $j$, each increase in $s+1-j$ corresponds to traveling up by $k$ boxes, and each row change decreases the box index by $q$.

We may assume that each connecting segment in the front projection lies in an arbitrarily small neighborhood of the center of its respective box, so the length of a chord may be estimated by counting the number of up-one-row, right-one-column steps needed to travel from Box $B(j)$ to Box $0$.  The box index decreases by $q$ with each up-one-row step, so $x_j$ counts the number of diagonal box lengths between the two boxes.

Finally, we note that an entire Reeb orbit measures $\frac{p}{k}$ diagonal  box lengths, so dividing $x_j$ by $\frac{p}{k}$ yields the length of the generator $a_j$ as a fraction of the orbital period. Since the chords $a_j$ and $b_j$ are complementary, the formula for the length of $b_j$ follows.
\end{proof}

We end this section with a brief comparison to Ng's resolution technique for Legendrian knots in $\mathbb{R}^3$.  In \cite{N}, Ng successfully reformulated Chekanov's algebra in terms of the front projection of  Legendrian knot. This can be mimicked for null-homologous knots in lens spaces, but the geometric constraints imposed by representing a non-trivial homology class prevent this approach from being directly applied to the general case of knots in lens spaces.  In the preceding section, we have instead tried to simplify the process of translating between different projections, identifying generators of the algebra with chords on the front projection but not computing the boundary map until after the Lagrangian projection is produced.  Although it is possible to describe the loops on a grid diagram which correspond to the boundary of a disc counted by the differential, this description is not particularly useful for computational purposes.  In the final section, however, we see that under special circumstances, such loops can play a useful role,

\section{Augmentations of $(\mathcal{A}(K_0), \partial)$}\label{sect:augs}
 In the previous section,  we developed the correspondence between special front and Lagrangian projections.   In this section, we apply these results to study the question of when $\mathcal{A}(K_0)$ has augmentations.  Our approach relies on the grading on the DGA, which we introduce in Section~\ref{sect:grad}.  We show that when $gcd(q-1,p)=1$, the existence of augmentations depends only on a subclass of words appearing in the boundary of the preferred generators.  These words can be described in terms of certain loops on a grid diagram for $\widetilde{K_0}$,  the preimage of $K_0$ in $S^3$. In Section~\ref{sect:augapp}, we determine the existence of augmentations of $\mathcal{A}\big(K(p, p-1, 2)\big)$ by analyzing the set of special loops.  The DGA may still be computed from the Lagrangian projection described in Section~\ref{sect:laggrid} when $gcd(q-1,p)>1$, but the computations in this section rely on the diagram having $h+v-1$ crossings and $h+v+1$ components of $S^2-\Gamma$. In the remainder of this section, we restrict to the $gcd(q-1,p)=1$ case, but we indicate which of the propositions generalize naturally.

\subsection{The grading}\label{sect:grad}
 As noted above, we will restrict to the case when $\gcd(q-1,p)=1$.  
 
 A \textit{capping path} for the generator $x$ is a path $\eta$ in $\Gamma$ which is smooth away from the crossing associated to $x$, and which has the further property that at this crossing, $\eta$ turns a corner around a quadrant labeled by $x^+$.    In the special Lagrangian diagrams, each $a$-type generator has two capping paths, and no $b$-type generator has a capping path. For each $a_j$, the capping paths positively bound  discs which fill a quadrant marked $a_j^+$; let $\eta$ denote the capping path whose associated disc  lies in the complement of the south pole of $S^2$.  Use this disc  to define a rotation number  $r(\eta)$ which counts the number of counter-clockwise rotations of the tangent vector $\eta'$ in the disc $S^2-\{\text{south pole}\}$.  Assuming that the strands of $\Gamma$ are orthogonal at each crossing, this rotation number takes values in $\mathbb{Z}-\frac{1}{4}$. 
   
Letting $w_N(\eta)$ denote the winding number of $\eta$  with respect to the north pole, a capping path is \textit{admissible} if $w_N(\eta) \equiv 0 \mod p$.
 
 \begin{definition}\label{def:grad}
If $a_j$ is a generator with an admissible capping path $\eta$, the grading of $a_j$ is given by
  \[ |a_j| =2 \lceil r(\eta) \rceil -2\frac{p-1}{p}w_N(\eta) -1 +4 n(D_{\eta}).\]
 The grading of $b_j$ is given by
 \[|b_j|=3-|a_j|.\]
 \end{definition}
 
Orient $K_0$, and denote by $r(K_0)$ and $n(K_0)$ the rotation number of $K_0$ and the defect of the disc bounded by $K_0$.  The gradings above are well-defined up to $2\big(r(K_0)\big)-4n(K_0)$ \cite{L}, \cite{S}.  

 \noindent\textbf{Remark:} The definition of an admissible capping path generalizes to the case $k>1$ by replacing $p$ with $\frac{p}{k}$; for the corresponding formulae for the grading, see \cite{L}. 

 \subsection{Special boundary discs}\label{sect:agen}
 
 We will use Fuchs's characterization of augmentations \cite{Fu}:

Given a homomorphism $\epsilon:\mathcal{A}\rightarrow \mathbb{ Z}_2$,  a disc which contributes a term to $\partial x$ is \textit{special} with respect to $\epsilon$ if $\epsilon(y_i)=1$ for every $y_i$ in $w(f,x)$.   A graded homomorphism $\epsilon$ is an augmentation of $(\mathcal{A}, \partial)$ if for each generator $x_j$, the number of special boundary discs is even. 

\begin{lemma}\label{lem:grad} Any graded augmentation of $\mathcal{A}(K(p,q,h))$ sends every $a$-type generator to zero. 
\end{lemma}

We defer the proof of this lemma in order to first explain why it it helpful.  The statement that every $a$-type generator vanishes under any graded augmentation implies that only boundary words written exclusively in $b$-type generators can be special.  Such words only appear in the boundary of $a$-type generators:

\begin{lemma}\label{lem:word} If $f$ is an admissible disc contributing a term to $\partial x$ and $w(f, x)$ is written exclusively in $b$-type generators, then $x$ is an $a$-type generator.  
\end{lemma}

\begin{proof}
In the Lagrangian diagram for $K_0$, corners which preserve the orientation of $K_0$ are marked with $a^+$ and $b^-$ labels.  If $\gamma$ is a path which traces out $f(\partial D^2)$, the orientation of $\gamma$ relative to $K_0$ is preserved at every $b^-$ corner.  This implies that if $w(f, x)$ is written only in $b$-type generators, then $x^+$ must also preserve the orientation as well.  
\end{proof}

Thus, the existence of an augmentation for $(\mathcal{A}, \partial)$ depends solely on the boundary of the $a$-type generators.  Equivalently, the only discs that need to be considered are those in which every corner of $f(D^2)$ fills a corner labeled with a $b^-$.  In Section~\ref{sect:loops}, we will show that such discs can be identified with loops on the special front projection of $\widetilde{K_0}\subset S^3$.

We end this section with a proof of Lemma~\ref{lem:grad}:
\begin{proof}
Orient  $K_0$ to bound a disc disjoint from the south pole.  Then the rotation number of $K_0$ is $h+v$.  

The defect of the disc bounded by $K_0$ is equal to its area, and this value may be bounded arbitrarily close to $\frac{v}{p}$ by making the connecting curves on the special front projection lie in a sufficiently small neighborhood of the centers of the boxes.  According to the remark after Definition~\ref{def:grad}, this implies that $(\mathcal{A}(K_0), \partial)$ is graded by a cyclic group of order $2(h+v)-4v=2|h-v|$.  

Recall that the $a$-type generators are exactly those with capping paths.  It follows from Definition~\ref{def:grad} that the grading of each $a_i$ is the sum of an odd integer and  $\frac{4v}{p}$.  In order for the grading of an $a$-type generator to be congruent to $0$, $\frac{4v}{p}$ would have to be equal to an odd integer, but this contradicts the assumption that $\gcd(q-1,p)=1$.  Thus any graded augmentation sends each $a_i$ to $0$.  
\end{proof}
 
\subsection{Boundary maps of $a$-type generators}\label{sect:loops}

As  the discussion in Section~\ref{sect:agen} suggests, we will determine whether $(\mathcal{A}(K_0), \partial)$ has any augmentations by studying the summands in $\partial a_i$ which could be special in the sense of \cite{Fu}.  The boundary of a disc counted by the differential consists of segments of the Lagrangian projection of $K_0$.  This lifts to a loop  $\gamma\subset L(p,q)$ consisting of segments of $K_0$ alternating with the Reeb chords which label the corners of the disc.  Any such $\gamma$ may be lifted further to a loop $\widetilde{\gamma}\subset S^3$ which consists of segments of $\widetilde{K_0}$ alternating with Reeb chords whose endpoints lie on $\widetilde{K_0}$.  We will characterize potential special boundary discs by studying the front projections of their associated $\widetilde{\gamma}$ curves.  

Just as Reeb chords in $L(p,q)$ are identified with their front projections, let $\widetilde{a_i}$ denote either a lift of the chord $a_i$ to $S^3$ or the front projection of that chord to the Heegaard torus.  Lifting preserves chord length: $\mathit{l}(a_i)=\mathit{l}(\widetilde{a_i})$.

\begin{definition}
An $N$-loop is an oriented simple closed curve on a special front projection for $\widetilde{K_0}$ which satisfies the following:
\begin{itemize}
\item the curve is homotopic to the loop $\theta_1=c$ for some $c \in [0,2\pi]$;
\item the curve alternates between $\widetilde{b_i^-}$ chords and segments of $K_0$ traversed left to right and top to bottom.
\end{itemize}
If a single $\widetilde{b_i^-}$ is replaced by its complementary $\widetilde{a_i^+}$, the resulting curve is called an \textit{$N_i$-loop}.

Similarly, an $S$-loop is a simple closed curve on a special front projection for $\widetilde{K_0}$ which satisfies the following:
\begin{itemize}
\item the curve is homotopic to the loop $\theta_2=c$ for some $c \in [0,2\pi]$;
\item the curve alternates between $\widetilde{b_i^-}$ chords and  segments of $K_0$ traversed bottom to top and right to left.
\end{itemize}
If a single $\widetilde{b_i^-}$ is replaced by its complementary $\widetilde{a_i^+}$, the resulting curve is called an \textit{$S_i$-loop}.

\end{definition}

\begin{figure}[h]
\begin{center}
\includegraphics[scale=.4]{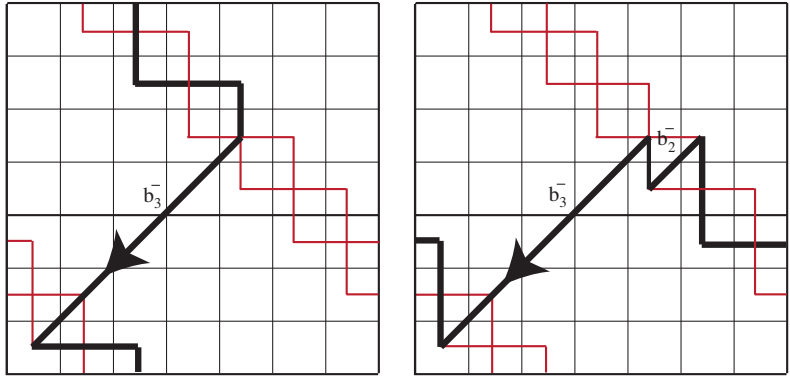}
\caption{Left: An $N$ loop for $K(7,6,2)$.  Right: An $S$ loop for $K(7,6,2)$.}\label{pic:NS}
\end{center}
\end{figure}

Each $N_i$- and each $S_i$-loop corresponds to some $F_{p,q}$-invariant $\widetilde{\gamma}$ in $S^3$, and equivalently, to some $\gamma \subset L(p,q)$.  The Lagrangian projection of this $\gamma$ is a closed curve in the Lagrangian projection of $K_0$, and the next proposition states that this loop bounds a disc counted by the differential.

\begin{proposition} Each $N_i$- or $S_i$-loop corresponds to a summand of $\partial a_i$, and the associated boundary word is written in the $b$-type chords in the $N_i$ or $S_i$ loop.  
\end{proposition}

\begin{proof}
Consider first an $S_i$ loop.  Because $\widetilde{a_i^+}$ replaced $\widetilde{b_i^-}$, the $S_i$-loop is  homotopic to the curve $\theta_1=c$ on the grid diagram.  The image of the $S_i$ loop in the front projection of the lens space lifts to a curve $\gamma\subset L(p,q)$, and the Lagrangian projection of $\gamma$ has winding number $p$ with respect to the south pole.  (This fact requires $gcd(q-1,p)=1$.) Thus, this curve  bounds an admissible disc $D$ in the Lagrangian projection.  We claim that the $a_i$-defect of $D$ is zero.  

\begin{figure}[h]
\begin{center}
\includegraphics[scale=.4]{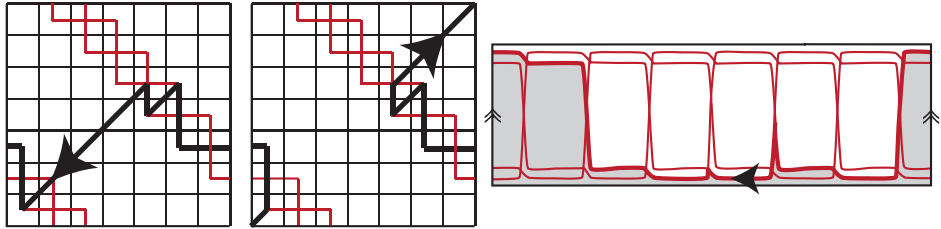}
\caption{Replacing the chord $b_3^-$ in an $S$-loop by $a_3^+$ yields an $S_3$ loops which gives the boundary of a disc with $a_3$-defect zero.  The lift of the disc with area $\frac{2}{7}$ to the Lagrangian projection of $\widetilde{K}\in S^3$ is shown on the right.}
\end{center}
\label{fig:replace}
\end{figure}

Recall that the defect of a region in $S^2 -\Gamma$ is the sum of the signed lengths of the preferred chords labeling the corners, minus the area of the region.  To compute the $a_i$-defect of $D$, replace each term $\mathit{l}(a_j)$ for $j\neq i$ in this sum by $\mathit{l}(a_j)-1=-\mathit{l}(b_j)$.  The proposition therefore follows if  $\mathit{l}(a_i)-\sum_{j\neq i} \mathit{l}(b_j)$  is equal to the area of  $D$.  As usual, we assume that the connecting segments on the special front projection lie in arbitrarily small neighborhoods of the centers of the boxes, so we treat the length of each chord as an integral multiple of $\frac{1}{p}$. This multiple may be computed by counting, with sign, the number of times the chord crosses a vertical line on the grid diagram for $\widetilde{K_0}$.  

The area of the large region region in $S^2-\Gamma$ has area approximately equal to $\frac{1}{p}$, and the multiplicity of this region in $D$ is equal to the number of times the boundary of $D$ traverses a horizontal segment of $\Gamma$ lying near the north pole.  The number of such segments equals the number of horizontal box-lengths traversed by the $S_i$-loop.  

Since the net horizontal displacement of the $S_i$ loop is zero, the horizontal displacement along chords is canceled by the horizontal displacement along the front projection of $\widetilde{K_0}$.  This shows that the $a_i$-defect of $D$ is zero, so the associated boundary word appears as a summand in $\partial a_i$.

The proof for an $N$-loop is similar, and the argument shows that if an $N$ or $S$ loop has $k$ chord segments, then it will correspond to $k$ distinct boundary discs.
\end{proof}

Counting $S$- and $N$-loops on the front projection of $\widetilde{K_0}$ shows that exactly two admissible capping paths  in the special Lagrangian diagram  contribute terms to the differential.
 
 \begin{proposition}\label{prop:cap} 
Let $j=p \mod (h+v)$.  One capping path for $a_j$ bounds a disc disjoint from the south pole which contributes a constant term to $\partial a_j$, and one capping path for $a_{h+v-j}$ bounds a disc disjoint from the north pole which contributes a constant term to $\partial a_{h+v-j}$.    
\end{proposition}

\begin{proof} 
The proof requires showing first that $a_j$ and $a_{h+v-j}$ both have admissible capping paths, and second, that these paths correspond to an $N_j$- and an $S_{h+v-j}$-loop, respectively.

Label the horizontal segments from $1$ to $h+v$, counting from the top down.  An admissible capping path for $a_i$ must complete $p$ rotations about the sphere; starting on the top horizontal segment,  each full rotation increases this index by one, counted modulo $h+v$.  If an admissible capping path for $a_i$ bounds a disc containing the north pole, then the path leaves the $i^{th}$ crossing to the right on the first horizontal segment, and enters $a_i$ from the left on the $i^{th}$ horizontal segment.  This implies that $i=p \mod h+v$.  

On the other hand, if an admissible capping path for $a_i$ bounds a disc containing the south pole, then the path leaves the $i^{th}$ crossing to the right on the $i+1^{th}$ segment and the path enters from the right on the $h+v^{th}$ segment.  This implies that $(i+1)+(j-1)=h+v \mod h+v$, so $i=h+v-j$.

\begin{figure}[h]
\begin{center}
\includegraphics[scale=.4]{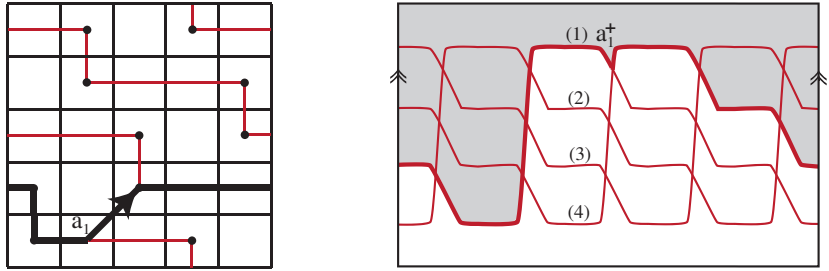}
\caption{Left: Replacing $a_1^+$ with $b_1^-$ in the front projection of the capping path would yield an $N$ loop.  Right: The capping path and the disc it bounds are shown lifted to a Lagrangian projection for $\widetilde{K}$; the parenthetical numbers identify the horizontal segments as they are numbered in the proof of Proposition~\ref{prop:cap}.}\label{fig:cap}
\end{center}
\end{figure}

Now lift the northern capping path for $a_j$ to $\widetilde{\gamma}\subset S^3$ and consider its front projection.  On the Heegaard torus, the path traverses the chord $a_j^+$ and then moves right and down along $K$ until reaching the other end of the chord.  This curve forms a loop homotopic to $\theta_2=c$ on the Heegaard torus, so replacing $a_j^+$ with $b_j^-$ yields a loop homotopic to $\theta_1=d$.  This is an $N$-loop, so the projection of $\widetilde{\gamma}$ was an $N_j$-loop and therefore bounded a disc counted by the differential.  Since the $N_j$-loop has a unique Reeb chord segment, the corresponding term in $\partial a_j$ is $1$.  

The proof for $a_{h+v-j}$ is similar.
\end{proof}

\subsection{Applications}
In this section we apply the results about $N$- and $S$-loops to examples where the number of generators is small.  Counting the number of possible loops allows us to determine whether or not $\mathcal{A}(K(p,p-1,h))$ has augmentations in the special cases $h=1$ and $h=2$.  A similar analysis should be possible for other values of $q$ and $h$, but the combinatorics involved in counting all possible $N$- and $S$-loops will be more complicated.

\subsubsection{Augmentations for $(\mathcal{A}(K(p,p-1,1)), \partial)$}

\textbf{Theorem~\ref{thm:main1}}\  
\textit{Let $K=K(p,p-1,1)$ be a Legendrian grid number one knot in $L(p,p-1)$ for $gcd(p-2,p)= 1$. Then the homology of $(\mathcal{A}(K),\partial)$ is a tensor algebra with two generators.  Furthermore, the map sending both generators of $\mathcal{A}(K)$ to $0$ is an augmentation. 
}

\begin{proof}
As shown in Theorem~\ref{thm:crossings},  the special Lagrangian diagram for $K(p,p-1,1)$ has only one intersection point.  Studying the front projection of $\widetilde{K_0}$ shows that the $a$ chord  is shorter than the $b$ chord, so only constant terms can appear in $\partial a$.  However, since both capping paths for the $a$ generator contribute constant terms, they cancel modulo two and $\partial a=0$.  Similarly, $\partial b=a+a=0$, so the entire algebra lies in the kernel of the differential.  Furthermore, both $\epsilon_1(b)=0$ and $\epsilon_2(b)=1$ are augmentations of $(\mathcal{A}(K(p,p-1,1)), \partial)$,  with $\epsilon_i(a)=0$ for $i=1,2$.  
\end{proof}

\subsubsection{Augmentations of $\big(\mathcal{A}\big(K(p,p-1,2)), \partial\big)$}\label{sect:augapp}
In the final section we consider the case $K=K(p,p-1,2)$. We will show the following: 

\noindent\textbf{Theorem~\ref{thm:main2}}
\textit{Let $K=K(p,p-1, 2)$ be a Legendrian grid number one knot in $L(p,p-1)$ for  $gcd(p-2,p)= 1$.
Then $(\mathcal{A}(K),\partial)$ has an augmentation if and only if $p\equiv 3 \mod 12$ or $p \equiv 9 \mod 12$.}\\
\\
The special Lagrangian diagram has three crossings, and Proposition~\ref{prop:cap} implies that the generators $a_1$ and $a_3$ will each have a capping path which bounds a boundary disc.  One of these discs contains the north pole, and we denote the corresponding generator by $a_N$.  Similarly, the other disc contains the south pole, and we denote the corresponding generator by $a_S$. 
 
\begin{lemma}\label{lem:bndry} The boundary of $a_N$ has no terms containing $b_S$ or $b_N$.  Similarly, the boundary of $a_S$ has no terms containing $b_S$ or $b_N$
\end{lemma}

\begin{proof}  Recall that at each crossing, the $a$-type generator is preferred.  Thus, if $f$ is a boundary disc whose boundary word is written entirely in $b$-type generators, then $n_{a_N}(f)=n(f(D^2))$.  The front projection of $K(p,p-1,2)$ shows that $\mathit{l}(a_N)=\mathit{l}(b_N)-1=\mathit{l}(b_S)-1$.  Any disc with positive corner $a_N^+$ that had a $b_N^-$ or $b_S^-$ corner would therefore have a negative defect.  This implies that no boundary word for $a_N$ can contain $b_S$ or $b_N$, and the argument for $a_S$ is identical.
\end{proof}

\begin{lemma}\label{lem:a2} If $\epsilon:\mathcal{A}\rightarrow \mathbb{Z}_2$ is a graded homomorphism such that $\epsilon(b_N)=\epsilon(b_S)$, then the generator $a_2$ will have an even number of special boundary discs.  
\end{lemma}

\begin{proof} Fix a representative of $a_2^+$ in the front projection of $\widetilde{K}$.  For each $N_2$-loop containing this chord, the reflection of the loop across the Reeb orbit containing the chord is an $S_2$-loop and vice versa.  This reflection interchanges $b_N$ and $b_S$,  and the hypothesis that  $\epsilon(b_N)=\epsilon(b_S)$ implies that the number of special boundary discs for $a_2$ is even.
\end{proof}

Since $b_N$ and $b_S$ do not appear in words written only in $b$-type generators in $\partial a_N$ and $\partial a_S$, we are free to choose $\epsilon(b_S)=\epsilon(b_N)$ without affecting the number of special boundary discs of $a_S$ and $a_N$.  Thus, Lemmas~\ref{lem:bndry} and \ref{lem:a2} together imply that the existence of an augmentation of $(\mathcal{A}, \partial)$ depends only on the number of words in $\partial a_N$ and $\partial a_S$ which are written solely in $b_2^-$. Note that the capping paths which bound boundary discs are special with respect to any homomorphism.

\begin{lemma}\label{lem:ns}
The number of boundary words of $a_N$ (respectively, $a_S$) associated to $S$- ($N$-) loops is odd unless $p\equiv \pm1\mod 12$. 
\end{lemma}

\begin{proof}
The chords $b_S$ and $b_N$ have the same image in the front projection of $K$, so they are represented by a total of  $p$ chords in the front projection of $\widetilde{K}$.  Pick one representative and fix this choice.  Note that the $S$-loops containing $b_N^-$ are reflections of the $N$-loops containing $b_S^-$ and vice versa.  Thus, it suffices to count only $S$-loops which contain $b_N^-$.

Begin tracing a path on the front projection of $\widetilde{K}$, beginning along $b_N^-$.  Moving only up and left along the image of $K$, it is not possible to form an $S$-loop without traveling along any other Reeb chords; if it were, the resulting loop would contradict the choice of chord.  Thus, in order to form an $S$-loop, the path must traverse some number of $b_2^-$ chords.  The maximal possible number of $b_2^-$ chords is the largest odd number less than or equal to $p-\mathit{l}(b_N)$. To see that this is the right value, observe that if the path traverses the last $b_2^-$ chord, it must end at $b_3$.  However, $b_3=b_N$ if and only if $p-\mathit{l}(b_N)$ is odd.  See Figure~\ref{fig:ends}.

\begin{figure}[h]
\begin{center}
\includegraphics[scale=.35]{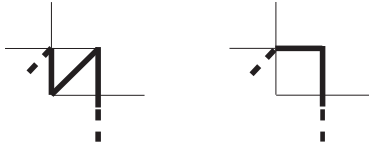}
\caption{Left: If $b_N=b_3$, the $S$ loop may traverse the last $b_2^-$.  Right: If $b_N=b_1$, the $S$ loop may not traverse the last $b_2^-$ }
\end{center}
\label{fig:ends}
\end{figure}

Locally, identify the two strands of the image of $\widetilde{K}$ as $A$ and $B$ as indicated in Figure~\ref{fig:aba}.  Each time the path traverses a $b_2^-$ chord, it switches between $A$ and $B$, and the total number of switches must be odd in order for the path to close up into an $S$-loop.  

\begin{figure}[h]
\begin{center}
\includegraphics[scale=.4]{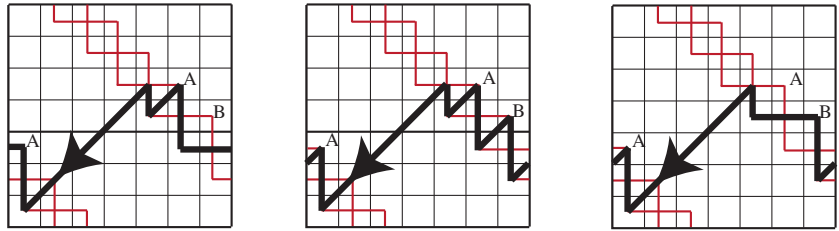}
\caption{In $L(7,6)$, the maximal number of switches is $3$.  The bold curves show the three possible $S$-loops which contain $b_N=b_3$.  Each $S$-loop corresponds to a boundary disc for $a_N$ which is disjoint from the north pole.}
\end{center}
\label{fig:aba}
\end{figure}

Labeling the $b_2^-$ chords by the strand they begin on, an $S$-loop is defined by an odd-length sequence of chords labeled alternately by $A$ and $B$. Thus each $S$-loop on a diagram with $2k+1$ switching chords corresponds to an alternating subsequence of a length $2k+1$ alternating sequence of $A$'s and $B$'s.  Note, too, that the subsequence must start with an $A$.  For $k=1$, the possible paths are given by the letters in bold: $AB\mathbf{A}, \mathbf{ABA}$, and $\mathbf{A}BA$ (Figure~\ref{fig:aba}).  For $k=2$, there are eight possible paths: 

\begin{align}
&\mathbf{A}BABA  & ABAB\mathbf{A}\notag\\
&AB\mathbf{ABA}  &\mathbf{ABA}BA\notag\\
&\mathbf{A}BA\mathbf{BA}  & \mathbf{AB}AB\mathbf{A}\notag\\
&AB\mathbf{A}BA &  \mathbf{ABABA}\notag
\end{align} 

Let $S(k)$ denote the number of $S$-loops which include $b_N$ on a diagram with $2k+1$ possible switching chords.  Picking the first and last $A$ chosen in a subsequence yields the following recursive formula:
\[S(k)=k+1+\sum_{i=1}^{k} iS(k-i), \text{ where } S(0)=1.\]

Expanding this yields
\begin{align} 
S(k)&=k+1+\sum_{i=1}^{k} iS(k-i)\notag\\
&=k+1+S(k-1)+\sum_{i=2}^{k} iS(k-i)\notag\\
&=(k+1)+(k-1+1)+ \sum_{j=1}^{k-1} jS(k-1-j)+\sum_{i=2}^{k} iS(k-i).\notag
\end{align}

We then reduce the previous equation modulo $2$:
\begin{align}
S(K)&\equiv1+ \sum_{j=1}^{k-1} jS(k-1-j)+\sum_{i=2}^{k} iS(k-i) \mod 2\notag\\
&\equiv 1+ \sum_{j=1}^{k-1} (2j+1)S(k-1-j) \mod 2\notag\\
 &\equiv 1+ \sum_{j=1}^{k-1} S(k-1-j) \mod 2\notag
\end{align}

Expanding the new relation and again reducing modulo two yields
\begin{align} 
S(k)&\equiv1+ \sum_{j=1}^{k-1} S(k-1-j) \mod 2\notag\\
&\equiv1+ S(k-2)+\sum_{j=2}^{k-1} S(k-1-j) \mod 2\notag\\
&\equiv1+ 1+\sum_{l=1}^{k-3}S(k-3-l)+\sum_{j=2}^{k-1} S(k-1-j) \mod 2\notag\\
&\equiv\sum_{l=1}^{k-3}S(k-3-l)+\sum_{j=2}^{k-1} S(k-1-j)\mod 2\notag\\
&\equiv S(k-3) \mod 2\notag
\end{align}

Computing the first few cases directly shows that $S(k)$ is odd except when $k\equiv 2\mod 3$.  The maximal number of switching chords is the greatest odd integer less than or equal to $p-\mathit{l}(b_N)=\frac{p-1}{2}$, so $k\equiv 2 \mod 3$ if and only if $p\equiv \pm 1 \mod 12$.  

\end{proof}

\begin{lemma}\label{lem:nn}
 The number of boundary words of $a_N$ associated to $N$-loops is odd unless $p\equiv 5 \mod  12$ or $p \equiv 7 \mod 12$.
\end{lemma}

\begin{proof}
This proof is similar to the previous one.  In this case the path which traverses no $b_2^-$ chords is an $N$-loop, and any other $N$-loop also traverses an number of $b_2^-$ chords.   If the maximal number of $b_2^-$ chords is $2k$, then $N(k)$ counts the number of even-length alternating subsequences beginning with $A$:
\[N(k)=1+\sum_{i=1}^k i N(k-i), \text{ where } N(0)=1 \text{ and } N(1)=2.\]
Expanding and reducing modulo two as above, this yields
\[N(k)\equiv N(k-3) \mod 2.\]
This value is odd unless $k\equiv 1 \mod 3$. The maximal number of possible $b_2^-$ chords is the greatest even number less than or equal to $\frac{p-1}{2}$, so $k \equiv 1\mod 3$ if and only if $p\equiv 5 \mod 12$ or $p \equiv 7 \mod 12$.

\begin{figure}[h]
\begin{center}
\includegraphics[scale=.4]{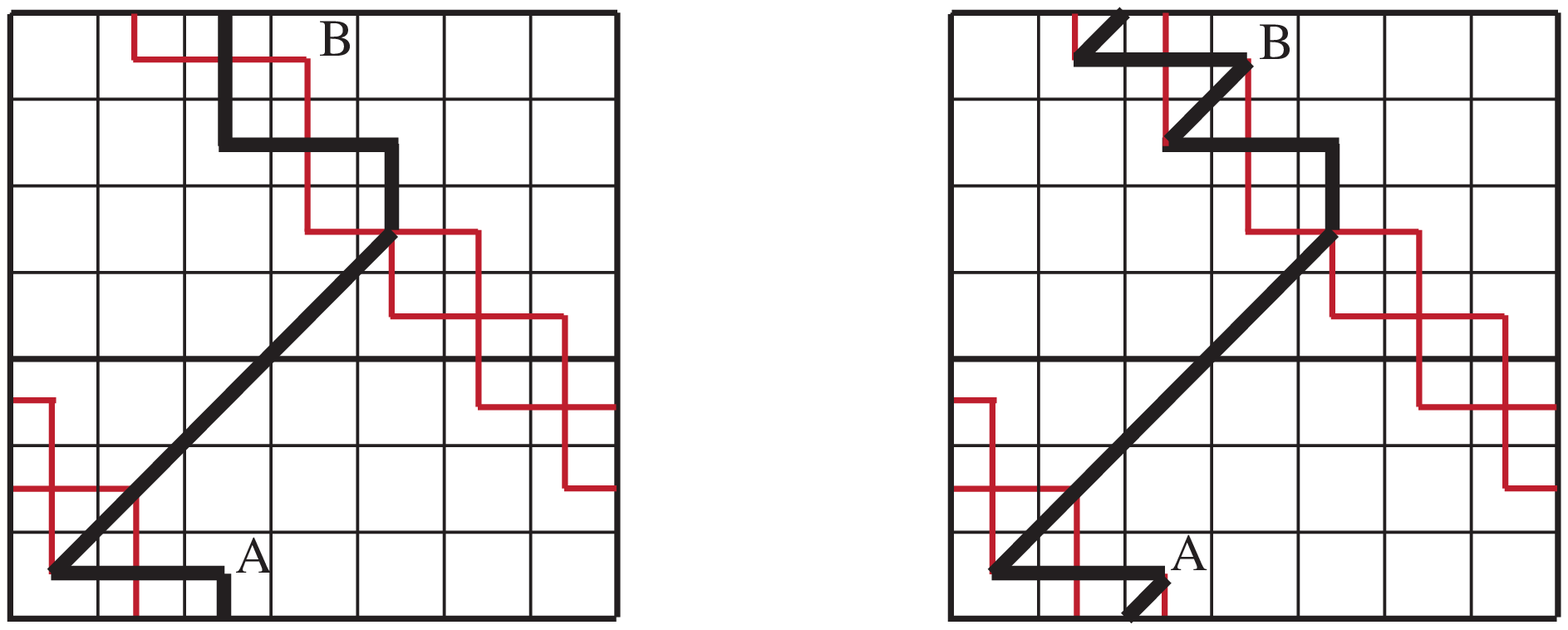}
\caption{The bold curves show the two possible $N$-loops which contain $b_N=b_3$.  Each $N$-loop corresponds to a boundary disc for $a_N$ which is disjoint from the south pole.}
\end{center}
\label{fig:abba}
\end{figure}
\end{proof}

When $p\equiv 3 \mod 12$ or $p \equiv 9 \mod 12$, then the total number of $N$- and $S$-loops containing $b_N$ is even.  Similarly, the total number of $S$- and $N$-loops containing $b_S$ is even.  Thus setting $\epsilon(b_2)=1$ implies an even number of special discs with respect to $\epsilon$.  Note that since a capping path always corresponds to a special disc, other values of $p$ cannot admit an even number of special discs for any $\epsilon$. Combining this with Lemmas~\ref{lem:bndry} and \ref{lem:a2} proves Theorem~\ref{thm:main2}.

\bibliographystyle{alpha}
\bibliography{Springerbib.bib}

\end{document}